\documentclass[a4paper, 11pt]{article}
\usepackage{amsmath, amssymb,amscd, amsthm}

\usepackage[mathscr]{eucal}
\usepackage{graphics}
\usepackage{fullpage}
\usepackage{url}
\usepackage{cancel}
\usepackage[margin=0.95in]{geometry}

\newcommand\cyr{%
 \renewcommand\rmdefault{wncyr}%
 \renewcommand\sfdefault{wncyss}%
 \renewcommand\encodingdefault{OT2}%
\normalfont\selectfont} \DeclareTextFontCommand{\textcyr}{\cyr}

\newtheorem{theorem}{Theorem}
\newtheorem{lemma}[theorem]{Lemma}
\newtheorem{corollary}[theorem]{Corollary}
\newtheorem{proposition}[theorem]{Proposition}
\newtheorem{remark}[theorem]{Remark}

\newtheorem{problem}[theorem]{Problem}

\long\def\symbolfootnote[#1]#2{\begingroup%
\def\thefootnote{\fnsymbol{footnote}}\footnote[#1]{#2}\endgroup}

\title{Sidonicity and variants of Kaczmarz's problem}

\author{Jean Bourgain\thanks{Partially supported by the NSF grant DMS-1301619.} \and Mark Lewko\thanks{Partially supported by a NSF Postdoctoral Fellowship, DMS-1204206.}}

\date{}

\begin{document}

\maketitle

\begin{abstract}We prove that a uniformly bounded system of orthonormal functions satisfying the $\psi_2$ condition: (1) must contain a Sidon subsystem of proportional size, (2) must satisfy the Rademacher-Sidon property, and (3) must have its five-fold tensor satisfy the Sidon property.  On the other hand, we construct a uniformly bounded orthonormal system that satisfies the $\psi_2$ condition but which is not Sidon. These problems are variants of Kaczmarz's Scottish book problem (problem 130) which, in its original formulation, was answered negatively by Rudin.  A corollary of our argument is a new elementary proof of Pisier's theorem that a set of characters satisfying the $\psi_2$ condition is Sidon.
\end{abstract}
\symbolfootnote[0]{2010 Mathematics Subject Classification 43A46, 42C05}

\section{Introduction}
Let $(\Omega,\mu)$ denote a probability space and  let $\{\phi_1,\phi_2,\ldots \}$ denote an orthonormal system (OS) of complex-valued functions on $\Omega$. A uniformly bounded OS is said to be Sidon with constant $\gamma$ if for all complex numbers $\{a_j\}$ one has
\begin{equation}\label{eq:Sidon}
 \sup_{x \in \Omega}| \sum_{j\in \mathbb{N}} a_j \phi_j(x)| \geq \gamma \sum_{j\in \mathbb{N}} |a_j|.
\end{equation}
Similarly, we will say that a system is Rademacher-Sidon with constant $\tilde{\gamma}$ if one has the inequality inequality
\begin{equation}\label{eq:RandSidon}
\int \sup_{x \in \Omega}| \sum_{j \in \mathbb{N}} r_j(\omega) a_j \phi_j(x)|d\omega \geq \tilde{\gamma} \sum_{j\in \mathbb{N}} |a_j|
\end{equation}
where $r_n$ denote independent Rademacher functions. Clearly if an OS is Sidon it is also Rademacher-Sidon. As we will see, the converse is not true. Sidonicity has typically been studied in the context of characters on groups. Indeed the reader may be more familiar with the terminology ``Sidon set" which refers to an OS comprised of a set of characters on a group. An introduction to the theory of Sidon sets may be found in \cite{GH} and \cite{LR}. The Sidon property \eqref{eq:Sidon}, however, can be studied in the more general setting of uniformly bounded systems. Our interest here will be the following question of S. Kaczmarz posed as Problem 130 in the Scottish book.
\begin{problem}\label{pro:Kacz}Let $\{\phi_n\}$ be a lacunary system of uniformly bounded orthogonal functions. Does there exists a constant $\gamma >0$, such that for every finite system of numbers $a_1,a_2,\ldots,a_n$ we have
$$ \max_{t} | a_1 \phi_1(t) + \ldots + a_n \phi_n(t) | \geq \gamma \sum_{j=1}^{n} |a_j|.$$
\end{problem}
A remark after the question defines a system to be lacunary if, for all $p>2$, there is a finite constant $M_p$ such that
 $$|| \sum_{j\in \mathbb{N}} a_j \phi_j ||_{L^{p}} \leq M_{p} \left( \sum_{j\in \mathbb{N}} |a_j|^2 \right)^{1/2}$$
holds for every sequence $\{a_n\}$.

In more modern language one might say that $\{\phi_n\}$ is a $\Lambda(p)$ system for every $p>2$. Such systems are sometimes referred to as $\Lambda(\infty)$. As we will explain, an example of Rudin provides a negative answer to this problem. The subsequent developments in the character setting suggest several natural relaxations, which we will study here.

Let us recall the development of the theory of Sidon sets/systems in the character setting. In 1960 Rudin introduced $\Lambda(p)$ sets and constructed a subset of the integers which is $\Lambda(\infty)$ but which is not Sidon. See Section 3.2 and Theorem 4.11 of \cite{Rudin}. This provides a negative answer to Kaczmarz's problem, although there is no evidence there that Rudin was aware of the problem's provenance. We will  briefly describe Rudin's construction. He first proved that a Sidon set must be $\Lambda(\infty)$ and, more restrictively, the set's $\Lambda(p)$ constants must satisfy $M_{p}\lesssim p^{1/2}$. From this he deduced that the size of the intersection of a Sidon set with an arithmetic progression of size $n$ must be $\lesssim \log n$. Rudin was then able to give a combinatorial construction of a set which (1) had too large of an intersection with a sequence of arithmetic progressions to be Sidon, yet 2) was $\Lambda(p)$ for all $p$. He established the second property by combinatorial considerations after expanding out $L^p$ norms in the case of even integer exponents.

On the other hand, much in the spirit of Kaczmarz's problem, Rudin asked if the stronger condition $M_{p}\lesssim p^{1/2}$ characterizes Sidon sets. In 1975 Rider \cite{Rider} proved that the Sidon condition \eqref{eq:Sidon} is equivalent to the (superficially) weaker Rademacher-Sidon condition \eqref{eq:RandSidon}. In 1978 Pisier \cite{Pisier} proved that Rudin's condition $M_{p}\lesssim p^{1/2}$ implies the Rademacher-Sidonicity property. Collectively these results show that Rudin's condition characterizes Sidonicity in the character setting. We note that both Rider's and Pisier's arguments make essential use of properties of characters. It is also worth noting that the first author \cite{BourRiesz} obtained a different proof of Pisier's theorem in 1983. The approach there, however, also relies on the homomorphism property of characters.

It is well known that Rudin's condition $M_{p} \leq C \sqrt{p}$ is equivalent to the condition that
$$|| \sum_{j} a_j \phi_j ||_{\psi_2} \leq C' \left( \sum_{j} |a_j|^2 \right)^{1/2} $$
where $|| \cdot ||_{\psi_2}$ is the Orlicz norm associated to the function $\psi_2(x) := e^{|x|^2}-1$. See, for example, Lemma 16 of \cite{Lewko}. We will refer to this condition as the $\psi_2(C')$ condition.

It is natural to ask how much of this theory can be generalized to arbitrary bounded orthonromal systems. Clearly Rudin's theorem that Sidonicity implies $\psi_2$ cannot hold in this generality. This can be seen by considering the direct product of a Sidon set/character system with a complete bounded orthonomal system. In the other direction, a natural relaxation of Kaczmarz's problem would be to ask if the $\psi_2$ condition implies Sidonicity in the case of general uniformly bounded orthonormal systems. Our first result is a construction of an OS that gives a negative answer to this question.
\begin{theorem}\label{thm:CE}For all large $n$, there exists a real-valued OS $\{\phi_0,\phi_1,\ldots, \phi_n \}$ with $n+1$ elements satisfying $||\phi_j ||_{L^{\infty}} \leq 7$  and satisfying the $\psi_2(C)$ condition with some universal constant $C$, and such that
$$ || \sum_{j=0}^{n} a_j \phi_j ||_{L^{\infty}} \lesssim \frac{1}{\sqrt{\log n}} \sum_{j=0}^{n}|a_j|$$
for some choice $\{a_j\}$.
\end{theorem}
This construction makes essential use of Rudin-Shapiro-type polynomials.

On the other hand, the following result provides a generalization of Pisier's theorem to general uniformly bounded orthonormal systems.
\begin{theorem}\label{thm:tensor}Let $\{\phi_j\}$ be a $\psi_2$ uniformly bounded OS. Then the OS obtained as a five-fold tensor, $\{\Phi_j := \phi_j\otimes \phi_j \otimes \phi_j \otimes \phi_j\otimes \phi_j  \}$, is Sidon.
\end{theorem}
Indeed using the homomorphism property, it easily follows that a $\psi_2$ system of characters must be Sidon. We will also show that the Sidonictiy (or Rademacher-Sidonicty) of a tensor system has the following implication for the system itself.
\begin{theorem}\label{thm:tensorImplies}If the k-fold tensor of an OS is Rademacher-Sidon then the system itself is Rademacher-Sidon.
\end{theorem}

It follows from Theorem \ref{thm:tensor} and Theorem \ref{thm:tensorImplies} that a  $\psi_2$ OS is Rademacher-Sidon. We will give several proofs of this fact. In fact, orthogonality beyond the $\psi_2$ condition, is not required.
\begin{theorem}\label{thm:RS}Let $\phi_1,\phi_2,\ldots$ denote a set of functions on a probability space $(\Omega,\mu)$ such that $||\phi_j||_{L^2}=1$ and satisfying the $\psi_2(C)$ condition. Then
\begin{equation}\label{eq:RS}
\int \sup_{x \in \Omega}| \sum_{j \in \mathbb{N}} r_j(\omega) a_j \phi_j(x)|d\omega \geq \tilde{\gamma} \sum_{j\in \mathbb{N}} |a_j|.
\end{equation}
with $\tilde{\gamma}:=\tilde{\gamma}(C)$.
\end{theorem}
This will be a corollary of the following result.
\begin{proposition}\label{prop:main}Let $\phi_1,\phi_2,\ldots,\phi_n$ be a system of functions satisfying the $\psi_2(C)$ condition and $||\phi_j||_{L^2}=1$. Let $x_1,x_2,\ldots,x_n$ denote (real or complex) vectors in a normed vector space satisfying $||x_j|| \leq 1$ and $\lambda_1,\lambda_2,\ldots,\lambda_n$ scalars. Then the estimate
$$\int \left|\left| \sum_{j=1}^{n} |\lambda_j| \phi_j(\omega) x_j \right| \right| d \omega \geq \beta \sum_{j=1}^n |\lambda_j|$$
implies
$$\int \left|\left| \sum_{j=1}^{n} \lambda_j r_j(\omega) x_j \right| \right| d \omega \geq \gamma \sum_{j=1}^n |\lambda_j|$$
for $\gamma:=\gamma(\beta, C)$.
\end{proposition}
Let us explain how Proposition \ref{prop:main} implies Theorem \ref{thm:RS}. By truncation it suffices to prove \eqref{eq:RS} for a finite system, as long as the bounds do not depend on the size of the system. We then have that
\begin{equation}\label{eq:L2LinftyTrick}
\sum_{j=1}^n |\lambda_j| = \int \sum_{j=1}^n |\lambda_j| |\phi_{j}(x)|^2 d x \leq \int  \left| \left| \sum_{j=1}^n |\lambda_j| \overline{\phi_{j}(x)} \phi_{j}(y)  \right| \right|_{L^{\infty}_{y}} d x.
\end{equation}
Using the $\psi_2(C)$ hypothesis, we may apply Proposition \ref{prop:main} to replace the functions $\{\overline{\phi_{j}(x)}\}$ with Rademacher functions and remove the absolute values. This gives us
$$\gamma \sum_{j=1}^n |\lambda_j| \leq \int  \left| \left| \sum_{j=1}^n \lambda_j r_j(\omega ) \phi_{j}(y)  \right| \right|_{L^{\infty}_{y}} d \omega  $$
which is Theorem \ref{thm:RS}.
Another variant of Kaczmarz's problem would be to ask if an appropriate hypothesis, such as the $\psi_2(C)$ condition, implies that a system contains a large Sidon subsystem. In this direction it follows that a finite uniformly bounded OS satisfying the $\psi_2(C)$ condition must contain a Sidon subsystem of proportional size. More precisely:
\begin{theorem}\label{thm:proportional}Let $\phi_1,\phi_2,\ldots,\phi_n$ be a system of functions satisfying $||\phi_j||_{L^2}=1$, $||\phi_j||_{L^\infty}\leq M$ and the $\psi_2(C)$ condition.
Then there exists a subset $S \subseteq [n]$ of proportional size $|S| \geq \alpha(C,M) n$ such that
$$\sup_{x \in \Omega}| \sum_{j\in S} a_j \phi_j(x)| \geq \gamma \sum_{j\in S} |a_j|.$$
where $\gamma = \gamma(C,M)$.
\end{theorem}
This is an immediate consequence of Theorem \ref{thm:RS} and the Elton-Pajor theorem.
\begin{theorem}\label{thm:EP}\emph{(Elton-Pajor)} Let $x_1,x_2,\ldots,x_n$ denote elements in a real or complex Banach space, such that $||x_i||\leq 1$. Furthermore, for Rademacher functions $r_1,r_2,\ldots,r_n$ assume that $ \gamma n \leq \int ||\sum_{i=1}^n r_i(\omega) x_i ||d \omega.$ Then there exists real constants $c:= c(\gamma)>0$ and $\beta := \beta(\gamma) >0$ and a subset $S \subseteq [n]$ with $|S| \geq c n$ such that
$$ \beta \sum_{j \in S}|a_j| \leq || \sum_{j \in S} a_j x_j || $$
for all complex coefficients $\{a_i\}_{i \in S}$.
\end{theorem}
One interesting consequence of Proposition \ref{prop:main} is that one may replace the Rademacher functions in the hypothesis of the Elton-Pajor Theorem with any complex-valued functions satisfying the $\psi_2(C)$ condition.

Our approach to Theorem \ref{thm:tensor} and Proposition \ref{prop:main} is rather elementary. The proofs proceed by showing that one may efficiently approximate a bounded system satisfying the $\psi_2(C)$ condition by a martingale difference sequence. Once one is able to reduce to a martingale difference sequence, one may apply Riesz product-type arguments. 

In Section  \ref{sec:5foldComplex}, we give an alternate approach to Proposition \ref{prop:main} based on more sophisticated tools from the theory of stochastic processes such as Preston's theorem \cite{Preston}, Talagrand's majorizing measure theorem \cite{TalagrandMM} and Bednorz and Lata\l{}a's \cite{BL} recent characterization of bounded Bernoulli processes. This approach yields a superior bound for the size of $\gamma(\beta, C)$ and allows for the following extension to more general norms.
\begin{theorem}\label{thm:RadMajor}Let $\phi_1,\phi_2,\ldots,\phi_n$ be a $\psi_2(C)$ system, uniformly bounded by $M$ and let $x_1,x_2,\ldots,x_n$ be vectors in a normed space $X$. Then
\begin{equation}\label{eq:radMaj}
\int || \sum_{j=1}^{n} \phi_j(\omega) x_j || d\omega \lesssim C M \int || \sum_{j=1}^{n} r_j(\omega) x_j || d\omega.
\end{equation}
In particular, one may take $\gamma(C,\beta) \gtrsim \beta\left( C \min\left(M, \sqrt{\log \frac{1}{\beta}  } \right)\right)^{-1}$ in Proposition \ref{prop:main} for $\psi_2(C)$ systems uniformly bounded by $M$.
\end{theorem}

Several problems related to this work are given in Section \ref{sec:problems}. The authors would like to thank Boris Kashin, Dan Mauldin, and Herv\'e Queff\'elec, for comments on an earlier draft of this manuscript.

\section{Proposition \ref{prop:main} without coefficients}\label{sec:Mart}
In order to present the proof as transparently as possible, we start by establishing Proposition \ref{prop:main} in the case that $\lambda_j=1$ for all $1\leq j \leq n$. We will then show how to adapt the proof to the case of general $\lambda_j$ in the next section. We start with the following lemma.

\begin{lemma}\label{lem:approx}Let $\phi_1,\phi_2,\ldots,\phi_n$ be real-valued functions on a probability space $(\Omega,\mu)$ such that
\begin{equation}\label{eq:approx1}
||\phi_j||_{L^2}=1 \hspace{.3cm}\text{and}\hspace{.3cm}||\phi_j||_{L^{\infty}} \leq C
\end{equation}
\begin{equation}\label{eq:approx2}
\left|\left| \sum_{j=1}^{n} a_j \phi_j \right|\right|_{\psi_2} \leq C \left( \sum_{j=1}^{n} |a_j|^2 \right)^{1/2}
\end{equation}
for all coefficients $\{a_j\}$. For $\epsilon >0$, there exists a subset $S \subseteq [n]$ such that $|S| \geq \delta(\epsilon, C) n $ and a martingale difference sequence $\{\theta_j\}_{j \in S}$ satisfying $||\theta_j||_{L^{\infty}} \leq C $ such that:
\begin{equation}\label{eq:approx3}
||\phi_j - \theta_j ||_{L^1} \leq \epsilon,
\end{equation}
and such that there exists an ordering of $S$, say $j_1,j_2,\ldots,j_n$, with
\begin{equation}\label{eq:approx4}
\mathbb{E} \left[\theta_{j_s} | \theta_{j_{s'}}, s' < s \right]=0.
\end{equation}
Moreover, one may take $\delta(\epsilon, M) \gtrsim  C^{-2} \epsilon^2 \left( \log \frac{C}{\epsilon} \right)^{-1}$.
\end{lemma}
\begin{proof}The functions $\theta_j$ will be discrete valued, taking at most $V$ values, with
\begin{equation}\label{eq:V}
V \lesssim \frac{C}{\epsilon}.
\end{equation}
More specifically, we will define
\begin{equation}\label{eq:approxForm}
\theta_{j_s} = \sum_{v=1}^{V} \sigma_{(s,v)} 1_{\Omega_{(s,v)}}
\end{equation}
where $\sigma_{(s,v)} \in \mathbb{R}$, $|\sigma_{(s,v)}|\leq C$, and $\{\Omega_{(s,v)} : v=1,\ldots, V \}$ a partition of $\Omega$. Letting $\mathcal{G}_{s}$ denote the set algebra generated by $\{\Omega_{(s',v)}: s'\leq s, v \leq V \}$ we clearly have
\begin{equation}\label{eq:GsBound}
|\mathcal{G}_{s}| \leq V^{s}.
\end{equation}
We will denote the atoms of $\mathcal{G}_{s}$ as $\{\Omega_{\alpha}^{(s)}\}$. Thus
$$\mathbb{E}\left[f |  \mathcal{G}_{s} \right] = \sum_{\alpha} \frac{\int_{\Omega_\alpha}f }{\mu(\Omega_\alpha)} 1_{\Omega_\alpha} $$
and
\begin{equation}\label{eq:L1martin}
\left|\left| \mathbb{E}\left[\phi_j |  \mathcal{G}_{s} \right] \right| \right|_{L^1} = \sum_{\alpha}  |\int_{\Omega_\alpha} \phi_j |.
\end{equation}
We will now construct $\theta_{j_s}$ by induction, with the base case being treated analogously to the induction step. Assume that $\theta_{j_s}$ has been constructed for all $s < t$ and let $\mathcal{J}_{t} \subset [n]$ be the set of $\mathcal{J}_{t} = \{j_s : s < t\}$.  We then have
\begin{equation}\label{eq:JtAve}
\sum_{j \in [n] \setminus \mathcal{J}_{t}} \left|\left| \mathbb{E}\left[\phi_j |  \mathcal{G}_{s} \right] \right| \right|_{L^1} \leq \sum_{\alpha} \sum_{j\in [n] \setminus \mathcal{J}_{t}} \left|\left<\phi_j, 1_{\Omega_\alpha} \right> \right|.
\end{equation}
For a fixed $\alpha$, using the $\psi_2(C)$ condition \eqref{eq:approx2}, we have
$$\sum_{j=1}^{n}\left|\left<\phi_j, 1_{\Omega_\alpha} \right> \right| \leq \max_{\epsilon_1,\ldots,\epsilon_n= \pm 1} \left| \left| \sum_{j=1}^{n} \epsilon_j \phi_j \right| \right|_{\psi_2} || 1_{\Omega_\alpha}||_{\psi_2^{*}}$$
\begin{equation}\label{eq:OrlHold}
\lesssim C\sqrt{n} |\Omega_\alpha | \left( \log \left(1 + \frac{1}{|\Omega_\alpha |} \right) \right)^{1/2}.
\end{equation}
Let $\delta$ be as given in the statement of the lemma and define
$$A := \{\alpha :  |\Omega_\alpha| > \delta V^{-t}  \} $$
$$A' := \{\alpha :  |\Omega_\alpha| \leq \delta V^{-t}  \}.$$
Using \eqref{eq:OrlHold}, the definition of $A$ and the inequality/hypothesis \eqref{eq:V} which states that $V \leq \frac{C}{\epsilon}$ and the inequality/hypothesis $t \leq \delta n$, we have that
$$\sum_{\alpha \in A} \sum_{j\in [n] \setminus \mathcal{J}_{t}} \left|\left<\phi_j, 1_{\Omega_\alpha} \right> \right|  \lesssim \sum_{\alpha \in A} C \sqrt{n}  |\Omega_\alpha| \left( \log \left(  \left(\frac{C}{\epsilon}\right)^{t} \delta^{-1} \right) \right)^{1/2}  $$
\begin{equation}\label{eq:ABound}\lesssim C \sqrt{n} t^{1/2} \left( \sqrt{ \log \left(\frac{C}{\epsilon} \right)}+ \sqrt{ \log \left( \frac{1}{\delta} \right) }\right) \lesssim C n \sqrt{\delta \log \left(\frac{C}{\epsilon} \right)}.
\end{equation}
On the other hand
\begin{equation}\label{eq:AprimeBound}
\sum_{j\in [n] \setminus  \mathcal{J}_{t}} \sum_{\alpha \in A'} \left|\left<\phi_j, 1_{\Omega_\alpha} \right> \right|  \lesssim n |\mathcal{G}_{s}| C \delta V^{-t} \lesssim C\delta n.
\end{equation}
It follows that $\sum_{j \in [n] \setminus  \mathcal{J}_{t}} \left|\left| \mathbb{E}\left[\phi_j |  \mathcal{G}_{s} \right] \right| \right|_{L^1} \leq C n \sqrt{\delta \log \left(\frac{C}{\epsilon} \right)} $, which allows us to find a $j_{t} \in [n] \setminus \mathcal{J}_{t}$ such that
$$\left|\left| \mathbb{E}\left[\phi_{j_t} |  \mathcal{G}_{t} \right] \right| \right|_{L^1} \lesssim C \sqrt{\delta \log \frac{C}{\epsilon}}.$$
We may now define $\theta_{j_t}$ to be an $\epsilon$-approximation (in $L^{\infty}$) to $\phi_{j_t} - \mathbb{E}\left[\phi_{j_t} |  \mathcal{G}_{t} \right] $ of the form \eqref{eq:approxForm}. We then have that
$$\left| \left| \theta_{j_t} -  \phi_{j_t} \right| \right| \lesssim \epsilon + C \sqrt{\delta \log \frac{C}{\epsilon}}$$
provided that $\delta \lesssim C^{-2}\epsilon^2 \left( \log \frac{C}{\epsilon} \right)^{-1}$ this quantity is $\lesssim \epsilon$. This completes the proof.
\end{proof}

Using Lemma \ref{lem:approx} we now are ready to prove Proposition \ref{prop:main}, again with the restrictions that $\lambda_j=1$ and $||\phi_j||\leq C$. Let $x_1,x_2,\ldots,x_n$ denote a sequence of vectors in a real or complex normed space $X$, and let $\phi_1,\phi_2,\ldots,\phi_n$ denote real-valued functions satisfying the hypothesis of Proposition \ref{prop:main}. We will return to the more general complex case shortly. We start by applying Lemma \ref{lem:approx} to obtain a martingale difference approximation ${\theta_j}$ to ${\phi_j }$. It clearly follows that $\prod_{j \in S} \left( 1+ \frac{\epsilon_j}{C} \theta_j \right) \geq 0$. Moreover, from the martingale difference sequence property \eqref{eq:approx3}, we have for all $\epsilon_j \in \{-1,1\}$, that
\begin{equation}\label{eq:ProductProb}
\int \prod_{j \in S} \left( 1+ \frac{\epsilon_j}{C} \theta_j(\omega) \right) d\omega = 1.
 \end{equation}
 Fix $\epsilon > 0$, then

$$ \int || \sum_{j=1}^{n} r_j(\omega) x_j ||d\omega  \geq  \int || \sum_{j=1}^{n} r_j(\omega) x_j ||   \prod_{j \in S} \left( 1+ \frac{r_j(\omega)}{C} \theta_j(\omega_2) \right) d\omega d\omega_2  \geq$$

\begin{equation}\label{eq:PhiXLower}\frac{1}{C} \int || \sum_{j \in S} \theta_j(\omega_2) x_j  ||   d\omega_2 \geq  \frac{1}{C} \left( \int || \sum_{j \in S} \phi_j(\omega) x_j ||d\omega - \epsilon |S|   \right).
\end{equation}

Returning to the case of complex $\phi_j$, let us split each function into real and imaginary parts as $\phi_j = \phi_{j}' + i \phi_j''$. From the assumption that
\begin{equation}\label{eq:Gamman}
\int || \sum_{j=1}^{n} \phi_j (\omega) x_j || d\omega > \beta  n,
\end{equation}
without loss of generality we may assume that
\begin{equation}\label{eq:GammaPrimen}
\int || \sum_{j=1}^{n} \phi_j' (\omega) x_j || d\omega > \frac{1}{2}\beta  n.
\end{equation}
Furthermore, we may find a subset $I \subseteq [n]$ such that $|I| \gtrsim \frac{1}{2}\beta n$ and such that, for each $S \subseteq I$, one has
\begin{equation}\label{eq:GammaI}
\int || \sum_{j \in S} \lambda_j \phi_j' (\omega) x_j || d\omega \gtrsim \frac{1}{2}\beta n.
\end{equation}
Since,
$$\int  || \sum_{j=1}^{n} r_j (\omega) x_j || d\omega  \geq \int  || \sum_{j \in I} r_j (\omega) x_j || d\omega $$
applying \eqref{eq:PhiXLower} with ${1,2,\ldots,n}$ replaced by $I$, we may lower bound this as
$$ \frac{1}{C} \left( \int || \sum_{j \in S} \phi_j'(\omega) x_j ||d\omega - \epsilon |S|   \right) $$
where $|S| \geq \delta(\epsilon, C) n$. Using that $|I| \gtrsim \frac{1}{2}\beta $, we have that
$$\frac{1}{C} \left(\frac{1}{2}\gamma - \epsilon \right)|S| \gtrsim  C^{-3} \beta ^{3} \left( \log \frac{1}{\beta } \right)^{-1} |I| \gtrsim  C^{-3} \beta ^{4} \left( \log \frac{1}{\beta } \right)^{-1} n.$$
This completes the proof of Proposition \ref{prop:main} in the case that $\lambda_j=1$.

\section{Proposition \ref{prop:main} with coefficients}\label{sec:MartCo}
We start with the following refinement of Lemma \ref{lem:approx}.
\begin{lemma}\label{lem:approxCo}Let $\phi_1,\phi_2,\ldots,\phi_n$ be real-valued functions on a probability space $(\Omega,\mu)$ satisfying
\begin{equation}\label{eq:approx1Co}
||\phi_j||_{L^2}=1 \hspace{.3cm}\text{and}\hspace{.3cm}||\phi_j||_{L^{\infty}} \leq C
\end{equation}
\begin{equation}\label{eq:approx2Co}
\left|\left| \sum_{j=1}^{n} a_j \phi_j \right|\right|_{\psi_2} \leq C \left( \sum_{j=1}^{n} |a_j|^2 \right)^{1/2}
\end{equation}
for all coefficients $\{a_j\}$. In addition, let $R>10$ be a large real constant, and let $\Lambda_1, \Lambda_2, \ldots, \Lambda_K$ be a partition of the functions $\phi_1,\phi_2,\ldots,\phi_n$ into sets satisfying
$$ |\Lambda_{k+1}| \geq R |\Lambda_k|.$$
For $\epsilon >0$, there exists subsets $S_k \subseteq \Lambda_k$ such that $|S_{k}| \geq \delta(\epsilon, C) |\Lambda_k| $ with the following properties. Letting $S= \cup_{k} S_k$, there exists a martingale difference sequence $\{\theta_j\}_{j \in S}$ satisfying $||\theta_j||_{L^{\infty}} \leq C $ such that:
\begin{equation}\label{eq:approx3Co}
||\phi_j - \theta_j ||_{L^1} \leq \epsilon.
\end{equation}
In addition, there exists an ordering of $S$, say $j_1,j_2,\ldots,j_n$, such that
\begin{equation}\label{eq:approx4Co}
\mathbb{E} \left[\theta_{j_{s'}} | \theta_{j_{s}}, s' < s \right]=0.
\end{equation}
Moreover, one may take $\delta(\epsilon, C) \gtrsim   C^{-2} \epsilon^2 \left( \log \frac{C}{\epsilon} \right)^{-1}$.
\end{lemma}
\begin{proof}We will construct each $S_k \subset \Lambda_k$ by induction, using the same approach used in the proof of Lemma \ref{lem:approx}. The case $k=1$ can be handled by a direct application of Lemma \ref{lem:approx}. We will assume throughout that $S_{k'} \subset \Lambda_{k'}$ has been constructed for $k'\leq k$, and that
\begin{equation}\label{eq:SkSize}
|S_{k'}| \asymp \delta(\epsilon, M)|\Lambda_{k'}|  \asymp   C^{-2} \epsilon^2 \left( \log \frac{C}{\epsilon} \right)^{-1} |\Lambda_{k'}|.
\end{equation}
Thus after constructing $S_{k'}$ for $k' \leq k$, we have, assuming $R>10$,
\begin{equation}\label{eq:UnionSize}
|\bigcup_{1\leq k' < k} S_{k'} | \leq \sum_{1 \leq k' < k} \delta |\Lambda_{k'}| \leq   \delta |\Lambda_{k}| \sum_{1 \leq k' < k} R^{-k'} \leq 2 R^{-1}  \delta |\Lambda_{k}|.
\end{equation}
Moreover, we will construct the elements of $S_k$ by induction as well. Let use denote the set of indices associated to $\Lambda_k$ as $\overline{\Lambda_{k}} := \{j \in [n] : \phi_j \in \Lambda_{k}\}$. Assume we have constructed $t-1$ elements so far. Then, as above, the set algebra $\mathcal{G}_t$ satisfies $|\mathcal{G}_{t}| \leq V^{t}$. In addition let $\mathcal{J}_{t}^{(k)} := \{ j_{s} \in [n] : s < t , \phi_{j_{s}} \in \Lambda_k \}$. As in the proof of Lemma \ref{lem:approx}, we have

$$\sum_{j \in \overline{\Lambda_{k}} \setminus \mathcal{J}_{t}} \left|\left| \mathbb{E}\left[\phi_j |  \mathcal{G}_{s} \right] \right| \right|_{L^1} \leq \sum_{\alpha \in \mathcal{G}_{t}} \sum_{j\in \overline{\Lambda_{k}} \setminus \mathcal{J}_{t}} \left|\left<\phi_j, 1_{\Omega_\alpha} \right> \right|$$
\begin{equation}\label{eq:JtAveCoA}
= \sum_{\alpha \in A} \sum_{j\in \overline{\Lambda_{k}} \setminus \mathcal{J}_{t}^{(k)} } \left|\left<\phi_j, 1_{\Omega_\alpha} \right> \right| + \sum_{\alpha \in A'} \sum_{j\in \overline{\Lambda_{k}} \setminus \mathcal{J}_{t}^{(k)}} \left|\left<\phi_j, 1_{\Omega_\alpha} \right> \right|.
\end{equation}
Following \eqref{eq:AprimeBound}, we have
$$  \sum_{\alpha \in A'} \sum_{j\in \overline{\Lambda_{k}} \setminus \mathcal{J}_{t}^{(k)}} \left|\left<\phi_j, 1_{\Omega_\alpha} \right> \right| \lesssim  |\Lambda_k| |\mathcal{G}_{t}| C \delta V^{-t} \lesssim C \delta |\Lambda_k|.$$

Similarly, following \eqref{eq:ABound}, we have
$$  \sum_{\alpha \in A} \sum_{j\in \overline{\Lambda_{k}} \setminus \mathcal{J}_{t}^{(k)}} \left|\left<\phi_j, 1_{\Omega_\alpha} \right> \right| \lesssim  C |\Lambda_{k}|^{1/2} t^{1/2} \sqrt{ \log \left(\frac{C}{\epsilon} \right)}.$$
As before, one may take $t$ as large as $\lesssim |\Lambda_k | \times C^{-2} \epsilon^{2} \left(\log \frac{C}{\epsilon} \right)^{-1}$. Selecting the implicit universal constant in the definition of $\delta(\epsilon, C)$ sufficiently small, we may assume that one may take, say, $t \leq 10 \delta(\epsilon, C)|\Lambda_k|$.

On the other hand, from \eqref{eq:UnionSize}, we have that $|\bigcup_{k' < k} S_{k'}| \leq R^{-1} \delta |\Lambda_k|$. We may thus find $S_{k} \subset \Lambda_k$ such that
$$|S_{k}| \geq  10 \delta |\Lambda_k| - 2R^{-1} \delta |\Lambda_k|  = (10 - 2R^{-1} ) \delta |\Lambda_k|  \geq \delta |\Lambda_k|.$$
This completes the proof.
\end{proof}

Next we record the following elementary observation, following Lemma 3 in \cite{BourRiesz}.
\begin{lemma}\label{lem:Coeff}Let $\phi_1,\phi_2,\ldots,\phi_n$ be functions uniformly bounded by $C$ satisfying the hypotheses of Lemma \ref{lem:approxCo}, and  $\lambda_1,\lambda_2,\ldots,\lambda_n$ be complex coefficients such that
$$\sum_{j=1}^{n}|\lambda_j|=1.$$
Given $\epsilon >0$ there exists a set $S \subseteq [n]$ and a martingale differences sequence $\theta_{j_1},\theta_{j_2},\ldots$ indexed by elements of $S$ satisfying \eqref{eq:approx3} and \eqref{eq:approx4}, such that
$$\sum_{j\in S}|\lambda_j| \gtrsim \delta(C,\epsilon) > 0.$$
\end{lemma}
\begin{proof}Let $R$ be the constant appearing in Lemma \ref{lem:approxCo} and $\delta:=\delta(C,\epsilon)$. Define
$$U_k = \{\phi_k : R^{-k} \geq |\lambda_k | > R^{-k-1}  \}$$
and $\overline{U}_k := \{ k \in \mathbb{N}: \phi_k \in U_k\}$ (we will use this convention of denoting an associated index set with an overline throughout the proof). Define $Z_{\text{e}}$ and $Z_{\text{o}}$ (respectively $\overline{Z}_{\text{e}}$ and $\overline{Z}_{\text{o}}$) as
$$ Z_{\text{e}} := \bigcup_{k \text{ even}} U_k \hspace{1cm} \text{and} \hspace{1cm} Z_{\text{o}} := \bigcup_{k \text{ odd}} U_k.$$
Since
$$ \sum_{j \in \overline{ Z}_\text{e}} |\lambda_j|  + \sum_{j' \in \overline{ Z}_{\text{o}} }|\lambda_{j'}| \geq 1.$$
We may find $Z \in \{ Z_{\text{e}},  Z_{\text{o}} \}$ satisfying
$$\sum_{k \in \overline{ Z}} |\lambda_k| \geq \frac{1}{2}.$$
Let $N$ denote the set of even (respectively odd) integers if $Z=Z_{\text{e}}$ (respectively $Z=Z_{\text{0}}$).
Next define $k_0=0$ and $k_{j+1} = \min\{k > k_j : |U_{k}| \geq R |U_{{k_j}}|, k \in N \}$. Taking $V_k = U_{j_k}$, we have $|V_{k+1}| \geq R |V_k|$, which allows us to invoke Lemma \ref{lem:approxCo} to obtain subsets $\Lambda_k \subset V_k$ such that $|\Lambda_k| \geq \delta |V_k|$ and satisfying the other conclusions of the Lemma \ref{lem:approxCo}.  We have
$$1 = \sum_{j \in \mathbb{N}} \sum_{k_j < k \leq k_{j+1}} \sum_{i \in \overline{U}_k} |\lambda_i| \leq \sum_{j \in \mathbb{N}} \sum_{k > k_j } R^{-2k_j+1} R \delta^{-1}|\Lambda_{k_j} |    $$
$$\lesssim \delta^{-1}R^2 \sum_{j \in \mathbb{N}} R^{-2k_{j}} |\Lambda_{k_j} | \lesssim \delta^{-1}R^2 \sum_{i \in \overline{Z} }|\lambda_i|. $$
Thus, letting $S = \bigcup_{j \in \mathbb{N}} \overline{\Lambda}_{k_j}$, we have
$$\delta R^{-2} \lesssim \sum_{i \in S} |\lambda_i| $$
which completes the proof.
\end{proof}

We now are ready to prove Proposition \ref{prop:main} with the added uniform boundedness assumption $||\phi_j||_{L^\infty} <C$. This assumption will be removed in the next section. By multiplying the system elements $\phi_j$ by unimodular complex numbers, it suffices to assume that the $\lambda_j$ are non-negative real numbers. As before, we start by assuming that $\phi_1,\phi_2, \ldots, \phi_n$ are  real-valued functions on a probability spaces satisfying the $\psi_2(C)$ condition. Let $\sum_{j=1}^{n}|\lambda_j| =1$, and let $S\subseteq [n]$ satisfy Lemma \ref{lem:Coeff} for a choice of $\epsilon >0$ to be specified later. Denoting the martingale difference approximations given by the lemma as $\{\theta_j\}$, we again have
\begin{equation}\label{eq:ProductProb2}
\int \prod_{j \in S} \left( 1+ \frac{\epsilon_j}{C} \theta_j(\omega) \right) d\omega = 1
 \end{equation}
for all $\epsilon_j \in \{-1,1\}$. Let $x_1,x_2,\ldots,x_n$ denote a sequence of vectors in a real or complex normed space $X$ and assume that $\phi_1,\phi_2,\ldots,\phi_n$ are real-valued functions satisfying the hypothesis of Proposition \ref{prop:main}. We then have that

$$ \int || \sum_{j=1}^{n} \lambda_j r_j(\omega)  x_j ||d\omega  \geq  \int || \sum_{j=1}^{n} \lambda_j r_j(\omega) x_j ||   \prod_{j \in S} \left( 1+ \frac{r_j(\omega)}{C} \theta_j(\omega_2) \right) d\omega d\omega_2  \geq$$

\begin{equation}\label{eq:PhiXLowerCo}\frac{1}{C} \int || \sum_{j \in S} \lambda_j \theta_j(\omega_2) x_j  ||   d\omega_2 \geq  \frac{1}{C} \left( \int || \sum_{j \in S} \lambda_j \phi_j(\omega) x_j ||d\omega - \epsilon\sum_{j \in S} \lambda_j  \right).
\end{equation}
As before, in the case of a complex system $\{\phi_j\}$ we will split each function into real and imaginary parts as $\phi_j = \phi_{j}' + i \phi_j''$. Given that
\begin{equation}\label{eq:Gamman2}
\int || \sum_{j=1}^{n} \lambda_j \phi_j (\omega) x_j || d\omega > \beta \sum_{j=1}^{n}|\lambda_j|,
\end{equation}
without loss of generality we may assume that
\begin{equation}\label{eq:GammaPrimen2}
\int || \sum_{j=1}^{n} \lambda_j \phi_j' (\omega) x_j || d\omega > \frac{1}{2}\beta \sum_{j=1}^{n}|\lambda_j|.
\end{equation}
Furthermore, we may find a subset $I \subseteq [n]$ with $ \sum_{j \in I} \lambda_j \gtrsim \frac{1}{2}\beta \sum_{j=1}^{n}\lambda_j$ and such that for each $S \subseteq I$ one has
\begin{equation}\label{eq:GammaI2}
\int || \sum_{j \in S}  \phi_j' (\omega) x_j || d\omega \gtrsim \frac{1}{2}\beta \sum_{j \in S} \lambda_j.
\end{equation}
Proceeding as before, applying \eqref{eq:PhiXLowerCo} with $[n]$ replaced by $I$ we have
$$\int  || \sum_{j=1}^{n} \lambda_j  r_j (\omega) x_j || d\omega  \gtrsim \int  || \sum_{j \in I} \lambda_j r_j (\omega) x_j || d\omega \gtrsim  \frac{1}{C} \left( \int || \sum_{j \in S} \lambda_j \phi_j'(\omega) x_j ||d\omega - \epsilon \sum_{j \in S} \lambda_j  \right) $$
where $\sum_{j \in S} \lambda_j \gtrsim \delta(\epsilon, C) \sum_{j \in I} \lambda_j  \gtrsim \frac{1}{2} \beta \delta(\epsilon, C) \sum_{j=1}^{n}\lambda_j$. Taking $\epsilon \lesssim \frac{\beta}{2}$, we may lower bound the quantity above by
$$C^{-1}\left(\frac{1}{2}\beta - \epsilon \right) \sum_{j \in S} \lambda_j \gtrsim C^{-1}\beta \delta\left(\frac{\beta}{4}, C\right) \gtrsim  C^{-3} \beta^{4} \left( \log \frac{1}{\beta} \right)^{-1} .$$
This completes the proof of Proposition \ref{prop:main}.

\section{Proposition \ref{prop:main} for unbounded systems}\label{sec:unbounded}
In the proof of Proposition \ref{prop:main} given in the previous section we assumed that the elements of the system were uniformly bounded by $C$. In this section we show that this condition may be removed.

 Let $x_1,x_2,\ldots x_n$ denote points in a real or complex normed space $X$, such that $||x_i||\leq 1$. Assume that
\begin{equation}\label{eq:PhiLowerBound}
\gamma \sum_{j=1}^{n}|\lambda_j| \leq \int ||\sum_{j=1}^{n} |\lambda_j| \phi_j(\omega) x_j  ||d\omega.
\end{equation}
Using the assumption that $\{\phi_j\}$ is a $\psi_2(C)$ system, we have $\mu \left[ |\phi_j| \geq y \right] \lesssim e^{-y^2/C^2}$. Thus
$$ \gamma \sum_{j=1}^{n} |\lambda_j|  \lesssim \int ||\sum_{j=1}^{n} \lambda_j \phi_j(\omega) x_j  ||d\omega \lesssim  \int ||\sum_{j=1}^{n} \lambda_j \phi_j(\omega) 1_{\{|\phi_j| \leq y \}} x_j  ||d\omega
+  \sum_{j=1}^{n}|\lambda_j| e^{-y^2/C^2}.$$
Thus,
$$ \gamma  \sum_{j=1}^{n} |\lambda_j| - C' \sum_{j=1}^{n} |\lambda_j| e^{-y^2/C^{2}} \lesssim  \int ||\sum_{j=1}^{n} \phi_j(\omega) y^{-1} 1_{\{|\phi_j| \leq y \}} x_j  ||d\omega.$$
Selecting $y \asymp   \sqrt{C^2 \log \left( \frac{1}{\gamma} \right)}$ we then have that
$$\gamma \sum_{j=1 }^{n} |\lambda_j|   \lesssim  \int||\sum_{j=1 }^{n}  \lambda_j  \phi_j(\omega)  1_{\{|\phi_j| \leq y \}} x_j  ||d\omega.$$
Now the truncated system ${\phi_j(\omega)  1_{\{|\phi_j| \leq y \}}}$ is uniformly bounded by $\sqrt{C^2 \log \left( \frac{1}{\gamma} \right)}$ and thus one may apply the uniformly bounded case of Proposition \ref{prop:main} proved in the previous section. This argument also shows how the second claim of Theorem \ref{thm:RadMajor} follows from the first claim.

\section{Five-fold real-valued tensor systems are Sidon}\label{sec:5foldReal}
In this section we prove Theorem \ref{thm:tensor}. For the sake of exposition, we prove the result for real-valued systems first. The complex case, which requires some additional technical details, will be presented in the next section.

\begin{theorem}\label{thm:tensorR}Let $\{\phi_j\}$ be a OS uniformly bounded by $C$ and satisfying the $\psi_2(C)$ condition. Then the OS obtained as a five-fold tensor, $\{\Phi_j := \phi_j\otimes \phi_j \otimes \phi_j \otimes \phi_j\otimes \phi_j  \}$, is Sidon.
\end{theorem}
\begin{proof}Let $\phi_1,\phi_2,\ldots,\phi_5$ denote independent copies of the system $\{\phi_i\}$ on probability spaces $\Omega_1,\Omega_2,\ldots,\Omega_5$, respectively. Furthermore let $\tilde{\Omega}=\otimes_{s=1}^{5} \Omega_s$ and let ${r_i^{(1)}},{r_i^{(2)}},{r_i^{(3)}},{r_i^{(4)}}$ denote independent Rademacher functions on a distinct probability space $\mathbb{T}$. For a fixed set of coefficients $\{a_i\}$ and $\epsilon>0$, applying Lemma \ref{lem:approxCo} gives a martingale difference sequence, $\{\theta_j\}$, with the following properties:
\begin{equation}\label{eq:RTSsize}
\sum_{i \in A} |a_i| \gtrsim C^{-2} \epsilon^2 \left( \log \frac{C}{\epsilon} \right)^{-1} \sum_{i=1 }^{n}  |a_i|,
\end{equation}
for all $i\in[n]$
\begin{equation}\label{eq:RTSbound}
||\theta_i||_{L^{\infty}} \leq C
\end{equation}
and
\begin{equation}\label{eq:RTSappr}
||\phi_i - \theta_i||_{L^1} \leq \epsilon.
\end{equation}
For $0< \delta <1$ and $\alpha_i \in [-1,1]$, define
$$\mu_{(\alpha,\delta)} := \int_{\mathbb{T}} \prod_{i \in A} \left(1 + \delta \alpha_i r_i^{(1)} \theta_i(x_1) \right) \prod_{i \in A} \left(1 + \delta \alpha_i r_i^{(2)} r_i^{(1)} \theta_i(x_2) \right) \prod_{i \in A} \left(1 + \delta \alpha_i r_i^{(3)} r_i^{(2)} \theta_i(x_3) \right)  \times $$
$$\prod_{i \in A} \left(1 + \delta \alpha_i r_i^{(3)} r_i^{(4)}\theta_i(x_4) \right) \prod_{i \in A} \left(1 + \delta \alpha_i r_i^{(4)}  \theta_i(x_5) \right)d\omega. $$
Expanding out the product, and defining $\nu_{S}(x) :=\prod_{i\in S} \theta_i(x) $, we see that
\begin{equation}\label{eq:muExpand}\mu_{(\alpha,\delta)}= \sum_{S \subseteq A} \delta^{|S|} \prod_{i\in S} \alpha_i \prod_{i\in S} \theta_i(x_1)\ldots \theta_i(x_5) =
 \sum_{S \subseteq A} \delta^{|S|} \prod_{i\in S} \alpha_i  \bigotimes_{j=1}^{5} \nu_{S}(x_{j}).
\end{equation}
Note the use of Rademacher functions in the definition of $\mu_{(\alpha,\delta)}$ leads to the elimination of certain terms involving products of the functions $\theta_i$'s in the expression above.  Assuming $\delta$ is sufficiently small depending on $C$ we clearly have that
\begin{equation}\label{eq:muL1}
||\mu_{(\alpha,\delta)}||_{L^1(\tilde{\Omega})}=1.
\end{equation}
To each subset $S \subseteq A$ we may associate a Walsh function on, say, the probability space $\mathbb{T}$ in the usual manner. In particular, let  $r_1,r_2,\ldots,r_m$ denote a system of Rademacher functions on $\mathbb{T}$ and form the associated Walsh system element associated to $S$ by $W_{S}(y) := \prod_{i\in S} r_i(y)$. Given $f$ such that $||f||_{L^\infty_x} \leq C$,  observe that
$$\left| \sum_{S \subseteq A} C^{-2|S|} W_{S}(y)\left<\nu_{S},f \right> \right| \leq \left|\left| \prod_{i \in A} \left(1+ C^{-2}r_i(y) \theta_i(x)\right) \right| \right|_{L^1_x} =1$$
where we have used $|C^{-2}\theta_i(x)f(x)| \leq 1$.  Since the function of $y$ defined by the expression on the left above is uniformly bounded by $1$ and thus has $L^2(\mathbb{T})$ norm at most $1$, Bessel's inequality gives us that
\begin{equation}\label{eq:muL2}
\sum_{S \subseteq A} C^{-4|S|} |\left< \nu_S, f \right>|^2 \leq 1.
\end{equation}
Using \eqref{eq:muExpand}, we have that
$$\left<\mu_{(\alpha,\delta)}, \Phi_i \right> = \delta \sum_{j\in A} \alpha_j |\left<\theta_j,\phi_i \right>|^5 +
\sum_{\substack{S \subseteq A \\ |S| \geq 2}} \delta^{|S|} \prod_{j \in S} \alpha_j |\left<\nu_S, \phi_i \right>|^5.
$$
We will estimate each of these terms separately. We start by estimating the second using \eqref{eq:muL2}. Provided $C^8\delta^2<1$, this gives
$$\sum_{\substack{S \subseteq A \\ |S| \geq 2}} \delta^{|S|} \prod_{j \in S} \alpha_j |\left<\nu_S, \phi_j \right>|^5 \leq C^{8}\delta^2.$$
We now consider the first term. By orthogonality and \eqref{eq:RTSappr} we have that
$$|\left<\theta_j, \phi_i \right>|^5 \leq |\left<\theta_j, \phi_i \right>|^2 \left(\left<\phi_j,\phi_i \right>+\epsilon \right)^{3}. $$
From this and \eqref{eq:muL2} we have, for $i \notin A$, that
$$\sum_{j\in A} |\alpha_j |\left<\theta_j,\phi_i \right>|^5 \leq \sum_{j \in A} |\left<\theta_j,\phi_i \right>|^5 \leq C^{4}\epsilon^3. $$
For $i \in A$, using again \eqref{eq:muL2}, we have that
\begin{equation}\label{eq:thetaphiAppr}
\left|\sum_{j\in A} \alpha_j \left|\left<\theta_j,\phi_i \right> \right|^5 - \alpha_i \left|\left<\theta_i,\phi_i \right>\right|^5 \right| \leq C^4 \epsilon^3.
\end{equation}
Finally we have
\begin{equation}\label{eq:thetaphiIP}
|\left<\phi_i,\theta_i \right>| \geq \left<\phi_i,\phi_i \right> - |\left<\phi_i,\phi_i-\theta_i \right> | \geq 1 - C\epsilon.
\end{equation}
Setting $\alpha_j = \text{sign}( a_j)$ for $j\in A$, the preceding estimates imply
$$\left<\sum_{i=1}^{n} a_i \Phi_i, \mu_{(\alpha,\delta)} \right> $$
$$\geq \delta \sum_{i \in A}|a_i| \left<\theta_i,\phi_i \right>^5 - \delta\left( \sum_{i\in A}|a_i| \right) \epsilon^3 - \delta \left(\sum_{i \notin A} |a_i| \right)\epsilon^3 - \delta^2 \sum_{i=1 }^{n} |a_i|.$$
Using \eqref{eq:thetaphiIP}, provided $\epsilon \lesssim C^{-1}$, we have that
$$\left<\sum_{i=1}^{n} a_i \Phi_i, \mu_{(\alpha,\delta)} \right>  \geq \delta \left(\frac{1}{2} \sum_{i \in A} |a_i| - C^4 \epsilon^3 \sum_{i=2}^{n}|a_i| - C^{8}\delta \sum_{i=1}^{n}|a_i| \right). $$
Recalling \eqref{eq:RTSsize}, we have that the quantity above is
$$\geq \delta \left(\frac{1}{2}C^2 \epsilon^2 \left( \log \frac{C}{\epsilon} \right)^{-1} -  C^4 \epsilon^3 - C^{8}\delta \right) \sum_{i=1}^{n}|a_i| .$$
The result follows by an appropriate choice of $\delta$ and $\epsilon$.
\begin{remark}If we replace the five-fold with a four-fold tensor in the preceding argument, the $\epsilon^3$ term in the previous display would be replaced by a factor of $\epsilon^2$ which would not be sufficient to conclude the proof.
\end{remark}

\section{Five-fold complex-valued tensor systems are Sidon}\label{sec:5foldComplex}
In this section we will develop a complex analog of the previous argument. This requires some additional notation. First let us denote the real and imaginary part of $\phi_j$ as $\phi_j=\phi_j'+i\phi_j''$. Given a sequence of complex scalars ${\alpha_j}$ let ${\theta_j'}$ and ${\theta_j''}$ denote respective martingale difference approximations satisfying \eqref{eq:RTSsize}, \eqref{eq:RTSbound}, and \eqref{eq:RTSappr}. Define real numbers $a_j$ and $b_j$ by $a_j+ib_j:=\text{sign}(\alpha_j)$.

Consider the $2^5$ 5-tuples of real and imaginary parts of system elements, ${\theta_j'}$ and ${\theta_j''}$. Call this set $T$. In a slight abuse of notation, it will be convenient to think of $T=\{',''\}^{5}$ as specifying a choice of either $\theta_j'$ or $\theta_j''$ in each of five coordinates. With this convention, for $t=(t_1,t_2,\ldots,t_5) \in T$ define $\nu_{S}^{(t_s)}= \prod_{i\in S} \theta_i^{(t_s)}(x)$. For each $t \in T$ we also define

\begin{equation}\label{eq:muExpandComplex}\mu_{(\beta^{(t)},\delta)}^{(t)}= \sum_{S \subseteq A} \delta^{|S|} \prod_{i\in S} \beta_i \prod_{i\in S} \theta_i^{(t_1)}(x_1)\ldots \theta_i(x_5)^{(t_5)} =
 \sum_{S \subseteq A} \delta^{|S|} \prod_{i\in S} \beta_i^{(t)}  \bigotimes_{s=1}^{5} \nu_{S}^{(t_s)}(x_{s}).
\end{equation}
As before, if $\delta \leq C^{-1}$ we have $||\mu||_{L^{\infty}} = 1$.
Next we will define $2\times2^5$ sequences of real numbers ${\beta^{(t)}_j}$ and ${\rho^{(t)}_j}$, indexed by $t \in T$. We let these sequences be specified by the relation
\begin{equation}\label{eq:defineAlphaBeta} (a_j+ib_j)\prod_{s=1}^{5} \left((\theta_j'(x_s)+i \theta_j''(x_s) \right)= \sum_{t \in T}\left( \beta_j^{(t)} \prod_{s=1}^{5} \theta_j^{(t_s)}(x_s) + i \rho_j^{(t)} \prod_{s=1}^{5} \theta_j^{(t_s)}(x_s) \right).
\end{equation}
We then have
$$\mu_{(\delta)} = \frac{1}{2^6} \left(\sum_{t \in T} \mu_{(\beta^{(t)},\delta)}^{(t)} + i\sum_{t \in T} \mu_{(\rho^{(t)},\delta)}^{(t)}\right) $$
$$=\frac{1}{2^6} \left( \sum_{t \in T} \sum_{S \subseteq A} \delta^{|S|} \prod_{j\in S} \beta_j^{(t)}  \bigotimes_{s=1}^{5} \nu_{S}^{(t_s)}(x_{s})  + i \sum_{t \in T} \sum_{S \subseteq A} \delta^{|S|} \prod_{j\in S} \rho_j^{(t)}  \bigotimes_{s=1}^{5} \nu_{S}^{(t_s)}(x_{s})   \right). $$
As before, for $\delta \leq C^{-1}$,  we have $||\mu_{(\delta)}||_{L^{\infty}} \leq 1$. Given $||f||_{L^\infty}\leq C$, we have that
$$\left| \sum_{S \subseteq A} C^{-2|S|} W_{S}(y)\left<\nu_{S}^{(t_s)},f \right> \right| \leq \left|\left| \prod_{i \in A} \left(1+ C^{-2}r_i(y) \theta_i^{(t_s)}(x)\right) \right| \right|_{L^1_x} =1$$
which implies
\begin{equation}\label{eq:muL2Complex}
\sum_{S \subseteq A} C^{-4|S|} \left|\left< \nu_{S}^{(t_s)}, f \right>\right|^2 \leq 1.
\end{equation}
Using \eqref{eq:defineAlphaBeta} one has
$$\left<\mu_{(\delta)}, \Phi_i \right> = \delta \sum_{j\in A}(a_j+i b_j)  \left< \theta'_j(x_s)+i\theta''_j(x), \phi_i (x_s)  \right>^5$$
\begin{equation}\label{eq:muInnerComplex} +
\sum_{\substack{S \subseteq A \\ |S| \geq 2}} \delta^{|S|} \prod_{j \in S} (a_j+i b_j) \left<\theta'_j(x_s)+i\theta''_j(x_s), \phi_i(x_s) \right>^5.
\end{equation}
It follows from \eqref{eq:muL2Complex} that, for sufficiently small $\delta$, $\sum_{\substack{S \subseteq A \\ |S| \geq 2}} \delta^{|S|} \left|\left<\theta'_j(x_s),\phi_i(x_s) \right>\right|^2 \leq C^{8}\delta^2$, and similarly with $\theta'_j(x_s)$ replaced by $\theta''_j(x_s)$. Combining this with the trivial estimate
$$\left< \theta'_j(x_s)+i\theta''_j(x), \phi_i (x_s)  \right>^5 \leq 2^5\left| \left< \theta'_j(x_s), \phi_i (x_s)  \right>\right|^5 +\left|\left< \theta'_j(x_s), \phi_i (x_s)  \right>\right|^5$$
allows us to estimate the second term on the right of \eqref{eq:muInnerComplex} as
$$\sum_{\substack{S \subseteq A \\ |S| \geq 2}} \delta^{|S|} \prod_{j \in S} (a_j+i b_j) \left<\theta'_j(x_s)+i\theta''_j(x_s), \phi_i(x_s) \right>^5 \lesssim  C^{8} \delta^2.$$
We now consider the first term on the right side of \eqref{eq:muInnerComplex}. By orthogonality we have that
\begin{equation}\label{eq:prodOrthComplex}
 \left|\left<\theta'_j(x_s)+i\theta''_j(x_s), \phi_i (x_s)  \right>\right|^5 \leq  \left|\left<\theta'_j(x_s)+i\theta''_j(x_s), \phi_i (x_s)  \right>\right|^2 \left(\left<\phi_j, \phi_i \right>+ 2 \epsilon \right)^{3}.
\end{equation}
Hence, if $i \notin A$ we have
$$\left| \sum_{j\in A}(a_j+i b_j)  \left< \theta'_j(x_s)+i\theta''_j(x), \phi_i (x_s)  \right>^5 \right| \lesssim C^4 \epsilon^3. $$
On the other hand, if $i \in A$,
$$\left| \sum_{j\in A}(a_j+i b_j)\left<\theta'_j(x_s)+i\theta''_j(x), \phi_i (x_s) \right>^5  -(a_i+i b_i)\left<\theta'_i(x_s)+i\theta''_i(x) ,\phi_i (x_s)\right>^5\right| \lesssim C^4 \epsilon^3. $$
Recalling that $(a_j+i b_j)=\text{sign}(\alpha_j)$ and letting $c_1,c_2,\ldots$ denote universal constants, using the expansion given in (\ref{eq:muInnerComplex}) we have that
$$\left< \sum_{i=1 }^{n}  \alpha_i \Phi_i, \mu_{(\delta)} \right> $$
$$\geq \delta \sum_{j \in A}|\alpha_j| \left<\theta_j,\phi_j \right>^5 - c_1\delta\left( \sum_{j\in A}|\alpha_j| \right) \epsilon^3 - c_1\delta \left(\sum_{j \notin A} |\alpha_j| \right)\epsilon^3 - c_1\delta^2  \sum_{j=1 }^{n} |\alpha_j|. $$
Using \eqref{eq:thetaphiIP}, which also holds in the complex case, provided $\epsilon$ is sufficiently small this gives that
$$\left<\sum_{i=1}^{n} a_i \Phi_i, \mu_{(\delta)} \right>  \geq \delta \left(\frac{1}{2} \sum_{i \in A} |\alpha_i| - c_2 C^4 \epsilon^3 \sum_{i=2}^{n}|\alpha_i| - c_2 C^{8}\delta \sum_{i=1}^{n}|\alpha_i| \right). $$
Recalling \eqref{eq:RTSsize}, we have that the quantity above is
$$\geq \delta \left(c_3 C^2 \epsilon^2 \left( \log \frac{C}{\epsilon} \right)^{-1} -  c_4 C^4 \epsilon^3 - c_4 C^{8}\delta \right) \sum_{j=1 }^{n}  |\alpha_j|.$$
Again, an appropriate choice of $\delta$ and $\epsilon$ completes the proof.
\end{proof}
\section{Tensor-Sidon implies Rademacher-Sidon}
The purpose of this section is to prove Theorem \ref{thm:tensorImplies}, namely:
\begin{proposition}Let $\{\phi_i\}$ denote a complex OS uniformly bounded by $M$ such that the $k$-fold tensored system $\{\otimes_{s=1}^{k} \phi_{s} \}$ is Rademacher-Sidon. Then $\{\phi_i\}$ has the Rademacher-Sidon property.
\end{proposition}
Let $k\geq 2$.  If $\{\otimes_{i=1}^{k}\phi_i\}$ is Rademacher-Sidon we have that
$$ \int \left| \left| \sum_{i=1}^{n} a_i g_i(\omega) \prod_{i=1}^{k} \phi_i(x_i)  \right| \right|_{L^{\infty}(\tilde{\Omega})}d\omega \geq c \sum_{i=1 }^{n} |a_i|.$$
We then claim that
$$   \int  \left| \left| \sum_{i=1}^{n} a_i g_i(\omega) \phi_i(x)  \right| \right|_{L^{\infty}(\Omega)}d\omega \gtrsim  \int \left| \left| \sum_{i=1}^{n} a_i g_i(\omega) \prod_{i=1}^{k} \phi_i(x)  \right| \right|_{L^{\infty}(\tilde{\Omega})} d \omega. $$
Recognizing that each side can be interpreted as the expectation of the supremum of a Gaussian processs, this inequality follows from the complex version of Slepian's comparison lemma (see Proposition \ref{prop:Cslepian} in the appendix) once one has established the following lemma.
\begin{lemma}In the notation above we have
$$\left(\sum_{i=1}^{n}|a_i|^2 \left| \prod_{s=1}^{k}\phi_i(x_s) - \prod_{s=1}^{k}\phi_i(x_s')   \right| \right)^{1/2} \leq \sqrt{k}\left(\sum_{s=1}^{k}\sum_{i=1}^{n}|a_i|^2 \left|\phi_i(x_s) - \phi_i(x_s') \right|^2 \right)^{1/2}.  $$
\end{lemma}
\begin{proof}Using the elementary inequality for complex numbers of modulus at most $1$, $\left|\prod_{i=1}^{k}a_i -  \prod_{i=1}^{k}b_i \right| \leq \sum_{i=1}^k |a_i - b_i|$, we have that
$$\left(\sum_{i=1}^{n}|a_i|^2 \left| \prod_{s=1}^{k}\phi_i(x_s) - \prod_{s=1}^{k}\phi_i(x_s')   \right| \right)^{1/2} \leq M^{k-1} \sum_{s=1}^{k} \left( \sum_{i=1}^{n}|a_i|^2 \left|\phi_i(x_s) - \phi_i(x_s') \right|^2 \right)^{1/2}.$$
An application of the Cauchy-Schwarz inequality shows the inequality above is
$$\leq\sqrt{k}\left(\sum_{s=1}^{k}\sum_{i=1}^{n}|a_i|^2 \left|\phi_i(x_s) - \phi_i(x_s') \right|^2 \right)^{1/2}.$$
\end{proof}It follows that
$$ \int \left| \left| \sum_{i=1}^{n} a_i g_i(\omega) \phi_i  \right| \right|_{L^{\infty}(\Omega)} d\omega \gtrsim_{M} c \sum_{i=1 }^{n} |a_i|. $$
One can replace the Gaussian random variables with Rademacher functions using a truncation argument (and the contraction principle), in a similar manner to the argument given in Section \ref{sec:unbounded}. Alternatively, one may apply Proposition \ref{prop:main}. This completes the proof.
\section{$\psi_2$ averages: Theorem \ref{thm:RadMajor}}\label{sec:stoch2}
In this section we present an alternate approach to Proposition \ref{prop:main} based on more sophisticated tools from the theory of stochastic processes. In order to state these results we recall some notation. Let $X$ denote a metric space with distance $d(t,s)$. Given a subset $\mathcal{E} \subseteq X$, we denote Talagrand's functional $\tau(\mathcal{E},d)$. We refer the reader to Chapter 2 of \cite{Tulb} (in particular Definition 2.2.19), where this quantity is denoted $\gamma_{2}(T,d)$, for a discussion and definition of this quantity. Moreover, we say that a stochastic process $X_{t}$  indexed by a subset of a metric space $\mathcal{E} \subseteq X$ is centered if $\int X_{t} d\mu =0$ for each $t \in \mathcal{E}$, and is subgaussian (with constant $C>0$) if it satisfies the inequality
\begin{equation}\label{eq:subGauss}
 \mu\left( |X_{t}-X_{s}| \geq \lambda \right) \leq C e\left( -\frac{\lambda^2}{C d(t,s)^2}\right).
\end{equation}
We may now recall Preston's theorem (see Theorem 3, in \cite{Preston}). A discussion/proof of the fact that the functional used in the statement of Theorem 3 of \cite{Preston} is equivalent to Talagrand's functional as defined in \cite{Tulb} can be found in \cite{Tmmwm}. Also note that in the centered case this result is presented as Theorem 2.2.18 in \cite{Tulb}.
\begin{proposition}\label{prop:Preston}Let $X_t$ be a subgaussian real-valued process indexed by elements of a metric space $X$ with distance $d(t,s)$. Then
$$\int \sup_{t \in T} |X_{t}| d\mu \lesssim C \tau(\mathcal{E},d).$$
\end{proposition}
On the other hand we have the following (see Lemma 3.2.6 in \cite{Tulb}) complex version of the Majorizing measure theorem:
\begin{proposition}\label{prop:complexMM}Let $X_{t}$ denote a complex-valued process such that $\Re X_{t}$ and $\Im X_{t}$ are Gaussian processes with respect to the metrics $d_{\Re}(s,t) = \left(\int | \Re X_s - \Re X_t |d\mu \right)^{1/2}$ and $d_{\Im}(s,t) = \left(\int | \Im X_s - \Im X_t |d\mu \right)^{1/2}$. Given the distance function $d(s,t) = \left(\int | X_s - X_t |d\mu \right)^{1/2}$, one has that
$$\tau(\mathcal{E},d) \lesssim \int \sup_{t \in \mathcal{E}} | X_t | d\mu.$$
\end{proposition}
Combining these results gives the following (a real-valued version of this inequality appears in the work of the first author \cite{BourPoly}):
\begin{corollary}\label{cor:gMajor}Let $\phi_{1},\phi_2,\ldots,\phi_n$ be a sequence of functions on a probability space $(\Omega,\mu)$ satisfying the $\psi_2(C)$ condition and let $\tilde{g}_1,\tilde{g}_2,\ldots,\tilde{g}_n$ denote a sequence of independent complex-valued Gaussian random variables. Furthermore, let $\mathcal{E} \subset \mathbb{C}^n$. Then,
$$ \int \sup_{\textbf{a} \in \mathcal{E} } \sum_{j=1}^{n} a_j \phi_j d\mu \lesssim C \int \sup_{\textbf{a} \in \mathcal{E} } \sum_{j=1}^{n} a_j \tilde{g}_j d\mu . $$
\end{corollary}
\begin{proof}
For $t =(t_1,\ldots,t_n) \in \mathcal{E}$, define
$$ X_{t} = \sum_{i=1}^{n} t_i \phi_i.$$
It follows from the $\psi_2(C)$ condition on the functions $\phi_1,\phi_2,\ldots,\phi_n$ (see Lemma 16 from \cite{Lewko}) that the process $X_t$ satisfies \eqref{eq:subGauss}. It easily follows that the real and imaginary parts of the process $X_t$ are subgaussian processes with respect to the same distance $d(t,s)$. In other words
\begin{equation}\label{eq:subGaussRI} \mu\left( |\Re X_{t}-\Re X_{s}| \geq \lambda \right) \leq C e \left( -\frac{\lambda^2}{C d(t,s)^2}\right) \hspace{.4cm} \text{and} \hspace{.4cm}  \mu\left( |\Im X_{t}-\Im X_{s}| \geq \lambda \right) \leq C e \left( -\frac{\lambda^2}{C d(t,s)^2}\right).
\end{equation}
It then follows from Proposition \ref{prop:Preston} that
\begin{equation}\label{eq:PresUB}
 \int \sup_{t \in \mathcal{E} } X_{t} \leq \int \sup_{t \in \mathcal{E} } |\Re X_{t}|d\mu + \int \sup_{t \in \mathcal{E} } |\Im X_{t}|d\mu \lesssim C \tau(\mathcal{E},d).
 \end{equation}
On the other hand, from Proposition \ref{prop:complexMM} we have
\begin{equation}\label{eq:MMlower}
\tau(\mathcal{E},d) \lesssim  \int \sup_{\textbf{a} \in \mathcal{E} } \sum_{j=1}^{n} a_j \tilde{g}_j d\mu.
\end{equation}
Combining \eqref{eq:PresUB} and \eqref{eq:MMlower} completes the proof.
\end{proof}
We will also require the recent result of Bednorz and Lata\l{}a \cite{BL} characterizing bounded Bernoulli processes. Given a subset $\mathcal{E} \subseteq \mathbb{C}^n$ we define the Bernoulli process
\begin{equation}\label{eq:Bprocess}
B(\mathcal{E}) := \int \sup_{t \in \mathcal{E}}\sum_{j=1 }^{n}  t_j r_j(\omega) d\omega.
\end{equation}
Let $G(T)$ denote the associated complex Gaussian process. In other words $G(T)$ is defined to be the quantity \eqref{eq:Bprocess} with the the Rademacher functions replaced by independent normalized complex-valued Gaussians. The theorem of Bednorz and Lata\l{}a states the following.
\begin{theorem}\label{thm:BL}Given a set $\mathcal{E} \subseteq \mathbb{C}^n$ with $B(\mathcal{E})< \infty$, there exists a decomposition $\mathcal{E} \subseteq \mathcal{E}_{1} + \mathcal{E}_{2}$ such that
\begin{equation}\label{eq:ellBL}
\sup_{t \in \mathcal{E}_{1}} \sum_{j=1 }^{n}  |t_j| \lesssim B(\mathcal{E})
\end{equation}
\begin{equation}\label{eq:gBL}
G(\mathcal{E}_{2}) \lesssim  B(\mathcal{E})
\end{equation}
where the implied constants are universal.
\end{theorem}
Strictly speaking, Bednorz and Lata\l{}a state their result for real-valued processes however the complex version follows by considering real and imaginary parts.
Theorem \ref{thm:RadMajor} will follow from the following proposition by taking $\mathcal{E} := \{ \lambda_i y(x_i) : y \in X^{*}, ||y|| \leq 1  \}$ where $X^{*}$ is the dual space of $X$.
\begin{proposition}Let $\phi_1,\phi_2,\ldots,\phi_n$ be a $\psi_2(C)$ system uniformly bounded by $M$ and $\mathcal{E} \subseteq \mathbb{C}^n$. Then
\begin{equation}
\int \sup_{t \in \mathcal{E}} | \sum_{j=1 }^{n}  t_j \phi_j(\omega) |d\omega   \lesssim M C  \int \sup_{t \in \mathcal{E}} |\sum_{j=1 }^{n}  t_j r_j(\omega)| d\omega.
\end{equation}
\end{proposition}
\begin{proof}Let $\mathcal{E}_1$ and $\mathcal{E}_2$ be as given in Theorem \ref{thm:BL}. We have that
$$\int \sup_{t \in \mathcal{E}} | \sum_{j=1 }^{n}  t_j \phi_j(\omega) |d\omega \leq \int \sup_{t \in \mathcal{E}_1} | \sum_{j=1 }^{n}  t_j \phi_j(\omega) |d\omega + \int \sup_{t \in \mathcal{E}_2} | \sum_{j=1 }^{n}  t_j \phi_j(\omega) |d\omega. $$
Applying Corollary \ref{cor:gMajor} and then Theorem \ref{thm:BL} we may bound the above quantity as
$$\lesssim C\int \sup_{t \in \mathcal{E}_1} |\sum_{j=1 }^{n}  t_j \tilde{g}_j(\omega) |d\omega + \sup_{t \in \mathcal{E}_2} \sum_{j=1 }^{n}  |t_j| \lesssim M C  \int \sup_{t \in \mathcal{E}} |\sum_{j=1 }^{n}  t_j r_j(\omega)| d\omega.$$
This completes the proof.
\end{proof}

\section{A Counterexample: Theorem \ref{thm:CE}}\label{sec:Example}
The purpose of this section is to prove Theorem \ref{thm:CE}. We start with the following elementary fact:

\begin{lemma}\label{lem:knpIn}Let $10 < n,p$ be positive real numbers. Then
$$\sqrt{\log n} n^{-1/p} \leq \sqrt{p}.$$
\end{lemma}
\begin{proof}The claim is equivalent to $\sqrt{p} n^{1/p} \geq \sqrt{\log n} $, or $ p n^{2/p} \geq \log n.$
Taking logarithms, this inequality is equivalent to $\log p + \frac{2}{p} \log n \geq \log \log n$. For a fixed $n$, the minimum of the left hand side occurs when $\frac{1}{p} - \frac{2}{p^2} \log n =0$, or $p = 2 \log n$. Thus we have
$$\log p + \frac{2}{p} \log n \geq \log \log n + \log 2 + 1 \geq \log \log n,$$
which establishes the claim.
\end{proof}
Next we estimate the $\Lambda(p)$ constant of the first $n$ elements of the Walsh system. Here, as above, $W_i$ denotes the $i$-th Walsh function on the unit unit interval $[0,1]$, which we'll denote as $\Omega_{1}$, in the standard (Paley) ordering.
\begin{lemma}\label{lem:HY}In the notation above, we have that
$$\frac{\sqrt{\log n}}{\sqrt{n}} ||\sum_{i=1}^{n} a_i W_i ||_{p} \lesssim \sqrt{p} \left( \sum_{i=1}^{n}|a_i|^2 \right)^{1/2}. $$
\end{lemma}
\begin{proof}By the H\"{o}lder's inequality we have that
$$\frac{\sqrt{\log n}}{\sqrt{n}}  ||\sum_{i=1}^{n} a_i W_i ||_{p} \leq \frac{\sqrt{\log n}}{\sqrt{n}} n^{1/p' - 1/2} (\sum_{i=1}^{n}|a_i|^2 )^{1/2} $$
$$\leq \sqrt{\log n} n^{-1/p} \left(\sum_{i=1}^{n} |a_i|^{2} \right)^{1/2}. $$
Applying Lemma \ref{lem:knpIn} completes the proof.
\end{proof}
For a fixed large $n$, let $\sigma_i \in \{-1,+1\}$ be chosen such that
\begin{equation}\label{eq:RSin}
 || \sum_{i=1}^{n} \sigma_i W_i ||_{L^{\infty}(\Omega_{1})} \leq 6 \sqrt{n}.
\end{equation}
In other words, $\sum_{i=1}^{n} \sigma_i W_i$ is a Walsh Rudin-Shapiro polynomial. The existence of the coefficients $\sigma_i$ is guaranteed, for instance, by Spencer's ``six standard deviations suffice" theorem (see the use of Theorem 1 in section 5 of \cite{Spencer}). Alternately, one can take $\{W_i\}_{n=1}^{N}$ to be the first $n$ exponentials and select $\sigma_i$ such that $\sum_{i=1}^{n} \sigma_i W_i$ is a classical Rudin-Shapiro polynomial (\cite{RudinS}). Next let $r_i$ denote independent Rademacher functions on $\Omega_{2}$. Furthermore define
$$\Psi := \left(1 + \frac{\log n}{n} \right)^{-1} \left( 1+ \frac{\log n}{n^2} \left(\sum_{i=1}^{n} r_i \right)^{2} \right),$$
where $\int_{\Omega} \Psi d\mu = 1.$ We now define a system of orthogonal functions $\phi_0,\phi_1,\ldots,\phi_n$ on the measure space $(\Omega,\Psi d\mu)$ where $\Omega = \Omega_{1} \times \Omega_{2}$. For $1 \leq i \leq n$ define
$$ \phi_i := \frac{1}{\sqrt{\Psi\left(1 + \frac{\log n}{n} \right)}} \left( r_i -  \frac{\sqrt{\log n}}{\sqrt{n}} \sigma_i W_i \right) $$
where $||\phi_i||_{L^\infty} \leq 1 \times (1+ \frac{\sqrt{\log n}}{\sqrt{n}}) \leq 2$. Next define
$$ \phi_0 :=  \frac{1}{\sqrt{\Psi\left(1 + \frac{\log n}{n} \right)}} \left( \frac{\sqrt{\log n}}{n} \sum_{i=1}^{n}r_i + \frac{1}{\sqrt{n}}\sum_{i=1}^{n} \sigma_i W_i \right).$$
Using that $\frac{1}{\sqrt{\Psi\left(1 + \frac{\log n}{n} \right)}} \leq 1$, $\frac{1}{\sqrt{\Psi\left(1 + \frac{\log n}{n} \right)}} \frac{\sqrt{\log n}}{n} \sum_{i=1}^{n}r_i \leq 1$ and $\left| \frac{1}{\sqrt{n}}\sum_{i=1}^{n} \sigma_i W_i \right| \leq 6$ by \eqref{eq:RSin}, for sufficiently larger $n$, we then have that
$$||\phi_0||_{L^\infty} \leq  1 + 6   \leq 7. $$
We now verify that this system satisfies orthonormality relations. For $1\leq i \leq n$
$$\int_{\Omega} |\phi_i|^2 \Psi d\mu = \int_{\Omega}  \frac{1}{\Psi} \left(1+ \frac{\log n}{n}\right)^{-1}\left( r_i -  \frac{\sqrt{\log n}}{\sqrt{n}} \sigma_i W_i \right)^2 \Psi d\mu $$
$$=\left(1+ \frac{\log n}{n}\right)^{-1}\left(1+ \frac{\log n}{n}\right) =1. $$
For $1\leq i, j \leq n$ and $i \neq j$ we have
$$\int_{\Omega} \phi_i \phi_j \Psi d\mu = \int_{\Omega}  \frac{1}{\Psi}\left(1+ \frac{\log n}{n}\right)^{-1} \left( r_i -  \frac{\sqrt{\log n}}{\sqrt{n}} \sigma_i W_i \right)\times \left( r_j -  \frac{\sqrt{\log n}}{\sqrt{n}} \sigma_j W_j \right) \Psi d\mu $$
$$=\left(1+ \frac{\log n}{n}\right)^{-1}\int_{\Omega}\left( r_i -  \frac{\sqrt{\log n}}{\sqrt{n}} \sigma_i W_i \right)\times \left( r_j -  \frac{\sqrt{\log n}}{\sqrt{n}} \sigma_j W_j \right) d\mu =0. $$
Next we consider $\phi_0$. We have
$$\int_{\Omega} |\phi_0|^2 \Psi d\mu = \left(1+ \frac{\log n}{n}\right)^{-1} \int_{\Omega}  \left( \frac{\sqrt{\log n}}{n} \sum_{i=1}^{n}r_i + \frac{1}{\sqrt{n}}\sum_{i=1}^{n} \sigma_i W_i \right)^2  d\mu $$
$$= \left(1+ \frac{\log n}{n}\right)^{-1}  \left(1+ \frac{\log n}{n}\right)=1. $$
For $1 \leq i \leq n$, we have

$$\int_{\Omega} \phi_0 \phi_i \Psi d\mu = \left(1+ \frac{\log n}{n}\right)^{-1} \int_{\Omega}  \left( \frac{\sqrt{\log n}}{n} \sum_{i=1}^{n}r_i + \frac{1}{\sqrt{n}}\sum_{i=1}^{n} \sigma_i W_i \right) \times \left( r_i -  \frac{\sqrt{\log n}}{\sqrt{n}} \sigma_i W_i \right) d\mu $$
$$= \left(1+ \frac{\log n}{n}\right)^{-1}  \left(\frac{\sqrt{\log n}}{\sqrt{n}}  - \frac{\sqrt{\log n}}{\sqrt{n}}    \right)=0. $$
This completes the verification that the construction gives a uniformly bounded OS. Next we verify the $\psi_2(C)$ condition.
\begin{lemma}The OS $\phi_{0},\phi_{1},\ldots,\phi_{n}$ satisfies the $\psi_2(C)$ condition for some fixed $C$ independent of $n$.
\end{lemma}
\begin{proof}Let $p\geq 2$, and $\sum_{i=1}^{n} |a_i|^2=1$. We have
$$|| \sum_{i=1}^{n} a_i \phi_i ||_{L^p} =\left( \int_{\Omega} \frac{|\sum_{i=1}^{n} a_i \phi_i|^p }{|\Psi|^{p/2}} \Psi d\mu \right)^{1/p} \leq \left( \int_{\Omega} |\sum_{i=1}^{n} a_i \phi_i|^p d\mu \right)^{1/p} $$
$$\lesssim || a_0 \psi_0 ||_{L^p(\Omega)} + ||\sum_{i=1}^{n} a_i r_i||_{L^p(\Omega_1)} +  \frac{\sqrt{\log n}}{\sqrt{n}} || \sum_{i=1}^{n} a_i \sigma_i W_i||_{L^p(\Omega_2)}. $$
Estimating the first term trivially, the second term using Khintchine's inequality, and the third using Lemma \ref{lem:HY} gives us that
$$|| \sum_{i=1}^{n} a_i \phi_i ||_{L^p(\Omega)} \lesssim \sqrt{p}.$$
This completes the proof.
\end{proof}
Finally, we show that these systems are not uniformly Sidon in $n$.
\begin{lemma}There exists coefficients $\{a_0,a_1,\ldots,a_n\}$ with unit $\ell^1$ norm, such that
$$|| \sum_{i=0}^{n} a_i \phi_i ||_{L^{\infty}(\Omega)} \lesssim \frac{1}{\sqrt{\log n}}.$$
\end{lemma}
\begin{proof}Set $a_{0} = -\frac{1}{\sqrt{\log n}}$ and $a_i=\frac{1}{n}$, for $1\leq i \leq n$. Then
$$ \left| -\frac{1}{\sqrt{\log n}} \psi_0 + \frac{1}{n}\sum_{i=1}^n a_i \phi_i \right|= \frac{1}{\sqrt{\Psi}}\left(1+ \frac{\log n}{n}\right)^{-1/2}  \times $$
 $$\left| - \frac{1}{n} \sum_{i=1}^{n} r_i + \frac{1}{n} \sum_{i=1}^{n} r_i - \frac{1}{\sqrt{n\log n}}\sum_{i=1}^{n} \sigma_n W_n + \frac{\log n }{n^{3/2} }\sum_{i=1}^{n} \sigma_n W_n \right|$$
$$\leq \left(\frac{1}{\sqrt{\log n}} + \frac{\log n}{n} \right) \left|   \frac{1}{\sqrt{n}}\sum_{i=1}^{n} \sigma_n W_n  \right| \lesssim \frac{1}{\sqrt{\log n}}$$
where we have used \eqref{eq:RSin}.
\end{proof}
This completes the proof of Theorem \ref{thm:CE}.

\section{Some related problems}\label{sec:problems}
In this section we record some problems raised by this work.

\begin{problem}\label{pro:proportion}Does there exists a constant $\gamma := \gamma(M,C,\epsilon)$ such that for any OS of size $n$, uniformly bounded by $M$, and satisfying the $\psi_2(C)$ condition, there exists a subset $A \subseteq [n]$ with $|A| \geq (1-\epsilon) n$ such that
$$|| \sum_{j \in A} a_j \phi_j ||_{L^{\infty}} \geq \gamma \sum_{j \in A}|a_j|?$$
\end{problem}

\begin{problem}\label{pro:2f}Is the two, three, or four-fold tensor of a uniformly bounded $\psi_2(C)$ orthonormal system Sidon?
\end{problem}

\begin{problem}\label{pro:Ave}Are all orthonormal $\psi_2(C)$ averages equivalent? In other words, if $\phi_1,\phi_2,\ldots,\phi_n$ are uniformly bounded, orthonormal and $\psi_2(C)$ can the inequality \eqref{eq:radMaj} be reversed?
\end{problem}

\begin{problem}\label{pro:Decomp}Is a uniformly bounded $\psi_2(C)$ OS a finite union of Sidon systems?
\end{problem}

\begin{problem}\label{pro:StrongElton}Let $x_1,x_2,\ldots,x_n$ be a set of unit vectors in a Banach space $X$. Assume that
$$ \gamma \sum_{i=1}^{n}|\lambda_i| \leq \int \left|\left| \sum_{i=1}^{n} \lambda_i r_i(\omega) x_i \right|\right|d\omega $$
for all scalar sequences $\{\lambda_i\}$ and $\gamma>0$. Does there exists a $M:=M(\gamma)$ and $\beta:=\beta(\gamma)>0$ such that $\{1,2,\ldots,n\}$ may be partitioned into $M$ sets $\{A_j\}_{j=1}^M$ such that
$$\beta \sum_{i \in A_j}|\lambda_i| \leq  \left|\left| \sum_{i\in A_j} \lambda_i x_i \right|\right|?$$
\end{problem}

\textbf{Note added:} G. Pisier has recently proven that the two-fold tensor of a $\psi_2(C)$ orthonormal system is Sidon providing an affirmative solution to Problem \ref{pro:2f}. See \cite{PisierNew}. In addition, Pisier has shown that that weaker hypothesis of Rademacher-Sidonicty implies that the four-fold tensor is Sidon. This raises the problem of deciding if Rademacher-Sidonicty implies that the three or two-fold tensor is Sidon. This would follow from an affirmative answer to Problem \ref{pro:Ave}.

\section{Appendix}
This appendix contains a number of results needed elsewhere in this paper which are well-known but for which we were unable to locate a proper reference.

First we need a complex variant of Slepian's comparison lemma. Let us recall the standard real version.
\begin{lemma}\label{lem:slepian}Let $X_t$ and $Y_t$ be real Gaussian process such that, for all $s,t$, one has
$$\mathbb{E}|X_s -X_t|^2  \leq \mathbb{E}|Y_s -Y_t|^2.$$
Then
$$\mathbb{E} \sup_{t \in T} X_t \leq \mathbb{E} \sup_{t \in T} Y_t.$$
\end{lemma}
We start by introducing some additional notation. Let $Z_t$ denote a complex Gaussian process and $Z_t'$ an independent copy of $Z_t$. Define
$$\tilde{Z}_{t}:= \Re[Z_t]+\Im[Z_t'].$$
For technical reasons the real-valued Gaussian process $\tilde{Z}_{t}$ is, at times, more convenient to work with than $Z_t$. The next lemma shows that the expectations of the suprmemum of these two processes are comparable.

\begin{lemma}\label{lem:CtoRcomp}In the notation above we have
$$ \mathbb{E} \sup_{t \in T}  |Z_t|   \lesssim  \mathbb{E}  \sup_{t \in T} | \tilde{Z}_t|   \lesssim  \mathbb{E}  \sup_{t \in T} | Z_t|.  $$
\end{lemma}
\begin{proof}Clearly $\mathbb{E} \sup_{t \in T} | Z_t|$ is greater than both $\mathbb{E}  \sup_{t \in T} | \Re Z_t|$ and $\mathbb{E}  \sup_{t \in T} | \Im Z_t|$. We claim that $\mathbb{E} \sup_{t \in T} | \tilde{Z}_t|$ majorizes both of these quantities as well. Indeed
$$\mathbb{E} \sup_{t \in T} \left| \tilde{Z}_t \right| = \mathbb{E}_{\omega_1} \mathbb{E}_{\omega_2}\sup_{t \in T}  \left| \Re[Z_t] + \Im[Z_t'] \right|$$
 $$\geq \mathbb{E}_{\omega_1} \sup_{t \in T} \left|  \Re[Z_t] + \mathbb{E}_{\omega_2} \Im[Z_t'] \right| \geq \mathbb{E}  \sup_{t \in T} \left| \Re Z_t \right|.$$
An analogous argument shows that $\mathbb{E} \sup_{t \in T} \left| \tilde{Z}_t\right| \geq \mathbb{E}  \sup_{t \in T} \left| \Im Z_t \right|$. We now have that
$$\mathbb{E} \sup_{t \in T} \left| Z_t \right| \leq \mathbb{E}  \sup_{t \in T} \left| \Re [Z_t] \right| +  \mathbb{E}  \sup_{t \in T} \left| \Im [Z_t] \right| \leq 2 \mathbb{E}  \sup_{t \in T} \left| \tilde{Z}_t \right|.$$
This establishes the first inequality. Similarly, using the definition of $X_t$, we have
$$ \mathbb{E} \sup_{t\in T} \left|\tilde{Z}_t \right| \leq \mathbb{E}  \sup_{t \in T} \left| \Re[ Z_t]\right| +  \mathbb{E}  \sup_{t \in T} \left| \Im [Z_t]\right| \leq 2 \mathbb{E}  \sup_{t \in T} \left| Z_t\right|.$$
This completes the proof.
\end{proof}

\begin{proposition}\label{prop:Cslepian}Let $Z_t$ and $W_t$ be Gaussian process such that
$$\mathbb{E}|Z_s -Z_t|^2  \leq \mathbb{E}|W_s -W_t|^2.$$
Then
$$\mathbb{E} \sup_{t \in T} Z_t \lesssim \mathbb{E} \sup_{t \in T} |W_t|.$$
\end{proposition}
\begin{proof}
By Lemma \ref{lem:CtoRcomp} we have
$$\mathbb{E} \sup_{t \in T} Z_t \lesssim \mathbb{E} \sup_{t \in T} |\tilde{Z}_t|.$$
Applying the Seplian's Lemma \ref{lem:slepian} to $\tilde{Z}_t$ and $\tilde{W}_t$ we have the above is
$$ \leq \mathbb{E} \sup_{t \in T} |\tilde{W}_t|. $$
Applying Lemma \ref{lem:CtoRcomp} we may further bound this by
$$\mathbb{E} \sup_{t \in T} |W_t|.$$
This completes the proof.
\end{proof}

\texttt{J. Bourgain}
\textit{bourgain@math.ias.edu}

\texttt{M. Lewko}
\textit{mlewko@gmail.com}

\end{document}